\documentclass[reqno,UKenglish, 12pt]{amsart}
\usepackage[margin=3cm]{geometry}
\usepackage{amsthm} 
\usepackage{amssymb}
\usepackage{color}
\usepackage{babel} 
\usepackage{hyperref}
\usepackage{amsmath, amsfonts, amssymb} 
\usepackage{graphicx}
\allowdisplaybreaks 
\usepackage[utf8]{inputenc} 

\title[Cheeger-like inequalities]{Cheeger-like inequalities for the largest eigenvalue of the graph Laplace Operator}
\date{}
\author{J\"urgen Jost}
\address{Max Planck Institute for Mathematics in the Sciences, D-04103 Leipzig, Germany}
\email{jjost@mis.mpg.de}
\author{Raffaella Mulas}
\address{Max Planck Institute for Mathematics in the Sciences, D-04103 Leipzig, Germany}
\email{raffaella.mulas@mis.mpg.de}

\newcommand{\R}{\mathbb{R}}

\DeclareMathOperator{\vol}{vol}

\DeclareMathOperator{\id}{Id}

\theoremstyle{plain}
\newtheorem{theorem}{Theorem}
\newtheorem{lem}[theorem]{Lemma}
\newtheorem{cor}[theorem]{Corollary}

\theoremstyle{definition}

\newtheorem*{definition}{Definition}

\newtheorem{ex}{Example}

\theoremstyle{remark}

\newtheorem{rmk}{Remark}

\usepackage{tikz}
\usetikzlibrary{positioning}
\tikzset{bluenode/.style={circle,fill=gray!50,minimum size=0.4cm,inner sep=0pt},}
\tikzset{rednode/.style={circle,fill=black!100,minimum size=0.4cm,inner sep=0pt},}

\begin{document}
	\maketitle
	\begin{abstract}We define a new Cheeger-like constant for graphs and we use it for proving Cheeger-like inequalities that bound the largest eigenvalue of the normalized Laplace operator.
		
		\vspace{0.2cm}
		
	\end{abstract}
	
	\section{Introduction}
	The (normalized) Laplace operator is a very powerful tool for the study of graphs, as its spectrum encodes important information \cite{Chung, JJspectrum, Hypergraphs, Bauer}. Here we consider unweighted and undirected (but oriented) graphs without loops, multiple edges and isolated vertices. For a fixed such graph on $n$ vertices, let's arrange the $n$ eigenvalues of the Laplace operator, counted with multiplicity, as
	\begin{equation*}
	\lambda_1\leq\ldots\leq\lambda_n.
	\end{equation*}
	We have $\lambda_1=0$, and the multiplicity of the eigenvalue $0$ equals the number of the connected components of the graph. Thus, 
	\begin{equation}\label{spec0}
	\lambda_2>0 
	\end{equation}
	if and only if the graph is connected. Henceforth, we shall only consider connected graphs. There is also a quantitative aspect. As we shall explain in more detail below, $\lambda_2$ estimates the coherence of the graph, that is, how different it is from a disconnected one.\\
	The largest eigenvalue, which is the main object of interest of this article,  satisfies
	\begin{equation*}
	\lambda_n\geq \frac{n}{n-1}
	\end{equation*}
	with  equality if and only if the graph is complete. For non-complete gaphs
	\begin{equation*}
	\lambda_n\geq  \frac{n+1}{n-1},
	\end{equation*}
	with  equality if and only if the graph either is obtained from a complete graph by removing a single edge or consists of two complete graphs of size $\frac{n+1}{2}$ that share a single vertex \cite{jmm}. In the other direction
	\begin{equation} \label{spec3}
	\lambda_n\leq 2   
	\end{equation}with equality if and only if at least one connected component of the graph is bipartite. \\
	For connected graphs, the first non--zero eigenvalue $\lambda_2$ is controlled both above and below by the Cheeger constant $h$, a quantity that measures how difficult it is to partition the vertex set into two disjoint sets $V_1$ and $V_2$ such that the number of edges between $V_1$ and $V_2$ is as small as possible and such that the \emph{volume} of both $V_1$ and $V_2$, i.e. the sum of the degrees of their vertices, is as big as possible. In particular,
	\begin{equation}\label{chee}
	\frac{1}{2}h^2\leq \lambda_2\leq 2h.
	\end{equation}Furthermore, there is an interesting characterization of $h$ obtained by writing $\lambda_2$ using the \emph{Rayleigh quotient} and then replacing the $L_2$--norm by the $L_1$--norm both in the numerator and denominator, as we shall see in Section \ref{sectionprel}.\newline
	
	In this paper, we want to explore an analogue of this for $\lambda_n$. In the same sense that by \eqref{chee}, $\lambda_2$ estimates how different a connected graph is from being disconnected, by \eqref{spec3}, $2-\lambda_n$ should quantify how different the graph is from being bipartite. One might therefore try to find the best (in a suitable sense) bipartite subgraph of our graph, because for a bipartite graph, the Rayleigh quotient that we shall discuss below is $2$, the maximal possible value. In fact, as it turns out, that subgraph can be quite small. More precisely, we shall introduce a new constant that is an analogue of the Cheeger constant in the sense 	    that it can be characterized by writing $\lambda_n$ using the \emph{Rayleigh quotient} and then replacing the $L_2$--norm by the $L_1$--norm both in the numerator and denominator. This constant is very simple,  
	\begin{equation*}
	Q:=\max_{\text{edges }v\sim w}\biggl(\frac{1}{\deg v}+\frac{1}{\deg w}\biggr).
	\end{equation*}
	Analogously to the Cheeger estimate \eqref{chee}, we shall  prove that it controls the largest eigenvalue $\lambda_n$ both above and below. Therefore, $Q$ is an analogue of the Cheeger constant for the largest eigenvalue. \\
	As we had explained above, $\lambda_2$ controls how different the graph is from a connected. Analogously, in view of \eqref{spec3}, one should expect that $2-\lambda_n$ measures the difference from a bipartite graph. 
	\newline
	
	Throughout the paper we shall also prove new general results of spectral graph theory that are useful in order to prove or discuss our main result.
	
	\subsection*{Structure of the paper}In Section \ref{sectionprel} we discuss the Laplace operator, the Cheeger constant, the dual Cheeger constant and the edge--Laplacian, as preliminaries to our work. In Section \ref{sectionmain}, and in particular in Theorem \ref{maintheo}, we present our main results and we prove them in Section \ref{sectionproof}. In Section \ref{sectionchoiceQ} we motivate the choice of $Q$, in Section \ref{sectionprecisionlower} we discuss the precision of our lower bound for $\lambda_n$ and finally in Section \ref{sectionprecisionupper} we discuss the precision of our upper bound. 
	
	\section{Preliminaries}\label{sectionprel}
	In this section we present some well--known results of spectral graph theory as preliminaries to our work; a general reference is \cite{Chung}.\newline
	
	From here on we fix a graph $\Gamma=(V,E)$ on $n$ vertices. We also fix an arbitrary orientation on $\Gamma$, that is, we see each edge as an arbitrarily ordered pair of its endpoints. Given $e=(v,w)\in E$, we say that $v$ is the \emph{input} of $e$ and $w$ is its \emph{output}. The fixed orientation is needed in order to do the computations when considering a function $\gamma: E\rightarrow \mathbb{R}$. However, the results are independent of the chosen orientation because, if one reverses the orientation of some edges, changing the sign of $\gamma$ on these edges leads to the same results. Therefore, the \emph{oriented edges} considered here should not be confused with \emph{directed edges}. Moreover, we shall use the notation $v\sim w$ for indicating (oriented) edges when input and output don't need to be distinguished.
	
	\subsection{Laplace operator and its eigenvalues} Let $\id$ be the $n\times n$ identity matrix, let $A$ be the \emph{adjacency  matrix} of $\Gamma$, let $D$ be the diagonal \emph{degree matrix} and let
	\begin{equation*}
	L:=\id-D^{-1}A
	\end{equation*}be the \emph{(normalized) Laplace operator}. 
	\begin{rmk}The Laplace operator considered in \cite{Chung} is $\mathcal{L}:=\id-D^{-1/2}AD^{-1/2}$. Since one can check that $L=\id-D^{-1/2}(\id-\mathcal{L})D^{1/2}$, the matrices $L$ and $\mathcal{L}$ are similar, therefore they have the same spectrum, including multiplicities, although the eigenfunctions can be different. 
	\end{rmk}

	By the \emph{Courant-Fischer-Weyl min-max principle}, we can write the eigenvalues 
	\begin{equation*}
	\lambda_1\leq\ldots\leq\lambda_n
	\end{equation*}
	of $L$ in terms of the \emph{Rayleigh quotient} \cite[pages 4 and 5]{Chung}. In particular,
	\begin{align*}
	\lambda_2&=\min_{f:V\rightarrow\mathbb{R}\text{ s.t. }f\neq 0,\,\sum_{v\in V}\deg v\cdot f(v)=0}\frac{\sum_{v\sim w}\biggl(f(v)-f(w)\biggr)^2}{\sum_{v\in V}\deg v\cdot f(v)^2}\\
	&=\min_{f:V\rightarrow\mathbb{R}\text{ non constant}}\max_{t\in\R}\frac{\sum_{v\sim w}\biggl(f(v)-f(w)\biggr)^2}{\sum_{v\in V}\deg v\cdot \bigl(f(v)-t\bigr)^2}
	\end{align*}and
	\begin{equation*}
	\lambda_n=\max_{f:V\rightarrow\mathbb{R},\,f\neq 0}\frac{\sum_{v\sim w}\biggl(f(v)-f(w)\biggr)^2}{\sum_{v\in V}\deg v\cdot f(v)^2}.
	\end{equation*}
	\begin{rmk}
		The condition $\sum_{v\in V}\deg v\cdot f(v)=0$ above is the \emph{orthogonality to the constants}. It comes from the fact that the constant functions are always eigenfunctions for $\lambda_1=0$ and, by the min-max principle, the eigenfunctions of the other eigenvalues must be orthogonal to them with respect to the scalar product $(f,g):=\sum_v\deg v\cdot f(v)\cdot g(v)$. The orthogonality to the constants is satisfied also by the eigenfunctions of $\lambda_n$, but in this case we don't need to specify it.
	\end{rmk}
	
	\subsection{Cheeger constant}For a connected graph $\Gamma=(V,E)$, the \emph{Cheeger constant} is defined as
	\begin{equation*}
	h:=\min_S\frac{|E(S,\bar{S})|}{\min\{\vol(S),\vol(\bar{S})\}}
	\end{equation*}where, given $\emptyset\neq S\subsetneq V$, $\bar{S}:=V\setminus S$, $|E(S,\bar{S})|$ denotes the number of edges with one endpoint in $S$ and the other in $\bar{S}$, and $\vol(S):=\sum_{v\in S}\deg(v)$.\newline
	
	The following  theorem \cite{dodziuk,alon} gives two important bounds for $\lambda_2$ in terms of $h$.
	\begin{theorem}\label{theoCheeger}
		For every connected graph,
		\begin{equation}\label{chee2}
		1-\sqrt{1-h^2}\leq \lambda_2\leq 2h.
		\end{equation}
	\end{theorem}Also, the following theorem \cite[Theorem 2.8 and Corollary 2.9]{Chung} shows the interesting relation between $h$ and $\lambda_2$ when, in the characterizations of $\lambda_2$ via the Rayleigh quotient, we replace the $L_2$--norm by the $L_1$--norm both in the numerator and denominator.
	\begin{theorem}\label{theocharCheeger}For every connected graph,
		\begin{equation*}
		h=\min_{f:V\rightarrow\mathbb{R}\text{ non constant}}\max_{t\in\R}\frac{\sum_{v\sim w}\bigl|f(v)-f(w)\bigr|}{\sum_{v\in V}\deg v\cdot \bigl|f(v)-t\bigr|}
		\end{equation*}and
		\begin{equation*}
		\frac{1}{2}h\leq\min_{f:V\rightarrow\mathbb{R}\text{ s.t. }f\neq 0,\,\sum_{v\in V}\deg v\cdot f(v)=0}\frac{\sum_{v\sim w}\bigl|f(v)-f(w)\bigr|}{\sum_{v\in V}\deg v\cdot \bigl|f(v)\bigr|}\leq h.
		\end{equation*}
	\end{theorem}
	\begin{rmk}Interestingly, the quantity
		\begin{equation*}
		\min_{f:V\rightarrow\mathbb{R}\text{ non constant}}\max_{t\in\R}\frac{\sum_{v\sim w}\bigl|f(v)-f(w)\bigr|}{\sum_{v\in V}\deg v\cdot \bigl|f(v)-t\bigr|}
		\end{equation*}that characterizes $h$ in Theorem \ref{theocharCheeger} is equal to the second smallest eigenvalue of the \emph{$1$--Laplacian} \cite{Chang2,Hein2,Hein,Chang,chang2016a}.
	\end{rmk}
	Our Theorem \ref{maintheo} below is an analogue of Theorem \ref{theoCheeger} and Theorem \ref{theocharCheeger} for the largest eigenvalue $\lambda_n$ in terms of our new constant $Q$. Before stating it, we shall discuss the dual Cheeger constant and the edge--Laplacian.
	\subsection{Dual Cheeger constant}In literature there is already a Cheeger-like constant that bounds the largest eigenvalue \cite{dual2, dual}. It is defined as
	\begin{equation*}
	\bar{h}:=\max_{\text{partitions } V=V_1\sqcup V_2 \sqcup V_3}\frac{|E(V_1,V_2)|}{\vol(V_1)+\vol(V_2)},
	\end{equation*}it is called the \emph{dual Cheeger constant} and it satisfies an analogue of \eqref{chee2}, 
	\begin{equation*}
	2\bar{h}\leq\lambda_n\leq 1+\sqrt{1-(1-\bar{h})^2}.
	\end{equation*}The two constants  $h$ and $\bar{h}$ are actually related to each other \cite{dual2}. For the dual Cheeger constant, however, there is no result analogous to Theorem \ref{theocharCheeger} \cite{chang2016b}. This  motivates the definition of  the new constant $Q$ that again bounds $\lambda_n$ and, additionally, satisfies an analogue of Theorem \ref{theocharCheeger}.
	\subsection{Edge--Laplacian}Associated to the Laplace operator there is also the \emph{edge--Laplacian}, defined as
	\begin{equation*}
	L^E:=\mathcal{I}^T D^{-1} \mathcal{I},
	\end{equation*}where $\mathcal{I}$ is the $|V|\times|E|$ \emph{incidence  matrix} of $\Gamma$.  Instead on acting on functions defined on the vertex sets, $L^E$ acts on functions defined on the edge set. It has the same non--zero spectrum of $L$ (i.e. the non--zero eigenvalues are the same, counted with multiplicity) and the multiplicity of the eigenvalue $0$ for $L^E$ equals the number of cycles of $\Gamma$ \cite{Hypergraphs}. We can therefore write the largest eigenvalue (that coincides for $L$ and $L^E$) also in terms of the Rayleigh quotient for functions on the edge set, by applying the min--max principle to $L^E$:
	\begin{align*}
	\lambda_n&=\max_{f:V\rightarrow\mathbb{R},\,f\neq 0}\frac{\sum_{v\sim w}\biggl(f(v)-f(w)\biggr)^2}{\sum_{v\in V}\deg v\cdot f(v)^2}\\
	&=\max_{\gamma:E\rightarrow\mathbb{R},\,\gamma\neq 0}\frac{\sum_{v\in V}\frac{1}{\deg v}\cdot \biggl(\sum_{e_{\text{in}}: v\text{ input}}\gamma(e_{\text{in}})-\sum_{e_{\text{out}}: v\text{ output}}\gamma(e_{\text{out}})\biggr)^2}{\sum_{e\in E}\gamma(e)^2}.
	\end{align*}In Section \ref{sectionmain} we shall present an analogue of Theorem \ref{theoCheeger} and Theorem \ref{theocharCheeger}, where:
	\begin{itemize}
		\item We look at $\lambda_n$ instead of $\lambda_2$;
		\item We use $Q$ instead of $h$;
		\item We use the point of view of the edge--Laplacian for considering the Rayleigh quotient and characterize $Q$.
	\end{itemize}
	\section{Main results}\label{sectionmain}
	Before stating our main theorem, let's recall that for a graph $\Gamma$ we have defined the new Cheeger--like constant 
	\begin{equation*}
	Q:=\max_{v\sim w}\biggl(\frac{1}{\deg v}+\frac{1}{\deg w}\biggr).
	\end{equation*}Let's also define the constant
	\begin{equation*}
	\tau:= \max_{v\sim w:\,\deg w\geq \deg v}\Biggl(\frac{(\deg w-\deg v+n)\cdot \deg v}{\deg v+\deg w}\Biggr).
	\end{equation*}

	\begin{theorem}\label{maintheo}
		For every graph,
		\begin{equation*}
		Q=\max_{\gamma:E\rightarrow\mathbb{R},\gamma\neq 0}\frac{\sum_{v\in V}\frac{1}{\deg v}\cdot \biggl|\sum_{e_{\text{in}}: v\text{ input}}\gamma(e_{\text{in}})-\sum_{e_{\text{out}}: v\text{ output}}\gamma(e_{\text{out}})\biggr|}{\sum_{e\in E}|\gamma(e)|}
		\end{equation*}and
		\begin{equation*}
		Q\leq \lambda_n\leq Q\cdot\tau.
		\end{equation*}
	\end{theorem}
	Observe that the characterization of $Q$ appearing in Theorem \ref{maintheo} equals the Rayleigh quotient we have used for writing $\lambda_n$ from the point of view of the edge--Laplacian, replacing the $L_2$--norm by the $L_1$--norm. Therefore, such a characterization is analogous to the one of $h$ in Theorem \ref{theocharCheeger}. We prove Theorem \ref{maintheo} in Section \ref{sectionproof}. Also, in Section \ref{sectionchoiceQ} we motivate the choice of $Q$, in Section \ref{sectionprecisionlower} we discuss whether  the lower bound appearing in Theorem \ref{maintheo} is sharp, and in Section \ref{sectionprecisionupper} we discuss the sharpness of the upper bound. 
	\section{Proof of the main results}\label{sectionproof}
	We split the statement of Theorem \ref{maintheo} into three parts. The first part, Lemma \ref{maintheo1}, contains the characterization of $Q$. The second part, Lemma \ref{maintheo2}, states that $Q\leq \lambda_n$. The third part, Lemma \ref{maintheo3}, states that $\lambda_n\leq Q\cdot \tau$.
	\subsection{Characterization of Q}
	\begin{lem}\label{maintheo1}For every graph,
		\begin{equation*}
		Q=\max_{\gamma:E\rightarrow\mathbb{R},\,\gamma\neq 0}\frac{\sum_{v\in V}\frac{1}{\deg v}\cdot \biggl|\sum_{e_{\text{in}}: v\text{ input}}\gamma(e_{\text{in}})-\sum_{e_{\text{out}}: v\text{ output}}\gamma(e_{\text{out}})\biggr|}{\sum_{e\in E}|\gamma(e)|}.
		\end{equation*}
	\end{lem}
	\begin{proof}In order to prove that
		\begin{equation*}
		Q\leq \max_{\gamma:E\rightarrow\mathbb{R},\,\gamma\neq 0}\frac{\sum_{v\in V}\frac{1}{\deg v}\cdot \biggl|\sum_{e_{\text{in}}: v\text{ input}}\gamma(e_{\text{in}})-\sum_{e_{\text{out}}: v\text{ output}}\gamma(e_{\text{out}})\biggr|}{\sum_{e\in E}|\gamma(e)|},
		\end{equation*}fix an edge $(v_1,v_2)$ that maximizes $\frac{1}{\deg v}+\frac{1}{\deg w}$ over all $(v,w)\in E$ and let $\gamma':E\rightarrow\mathbb{R}$ be $1$ on $(v_1,v_2)$ and $0$ otherwise. Then, 
		\begin{align*}
		Q&=\frac{1}{\deg v_1}+\frac{1}{\deg v_2}\\
		&=\frac{\sum_{v\in V}\frac{1}{\deg v}\cdot \biggl|\sum_{e_{\text{in}}: v\text{ input}}\gamma'(e_{\text{in}})-\sum_{e_{\text{out}}: v\text{ output}}\gamma'(e_{\text{out}})\biggr|}{\sum_{e\in E}|\gamma'(e)|}\\
		&\leq \max_{\gamma:E\rightarrow\mathbb{R},\,\gamma\neq 0}\frac{\sum_{v\in V}\frac{1}{\deg v}\cdot \biggl|\sum_{e_{\text{in}}: v\text{ input}}\gamma(e_{\text{in}})-\sum_{e_{\text{out}}: v\text{ output}}\gamma(e_{\text{out}})\biggr|}{\sum_{e\in E}|\gamma(e)|}.
		\end{align*}

		Let's now prove that
		\begin{equation*}
		Q\geq \max_{\gamma:E\rightarrow\mathbb{R},\,\gamma\neq 0}\frac{\sum_{v\in V}\frac{1}{\deg v}\cdot \biggl|\sum_{e_{\text{in}}: v\text{ input}}\gamma(e_{\text{in}})-\sum_{e_{\text{out}}: v\text{ output}}\gamma(e_{\text{out}})\biggr|}{\sum_{e\in E}|\gamma(e)|}.
		\end{equation*}
		Let $\hat{\gamma}:E\rightarrow\mathbb{R}$, $\hat{\gamma}\neq 0$ be a maximizer for
		\begin{equation*}
		\frac{\sum_{v\in V}\frac{1}{\deg v}\cdot \biggl|\sum_{e_{\text{in}}: v\text{ input}}\gamma(e_{\text{in}})-\sum_{e_{\text{out}}: v\text{ output}}\gamma(e_{\text{out}})\biggr|}{\sum_{e\in E}|\gamma(e)|}
		\end{equation*}such that, without loss of generality, $\sum_{e\in E}|\hat{\gamma}(e)|=1$. Then,
		\begin{align*}
		Q&=\max_{v\sim w}\biggl(\frac{1}{\deg v}+\frac{1}{\deg w}\biggr)\\
		&=\Biggl(\max_{v\sim w}\biggl(\frac{1}{\deg v}+\frac{1}{\deg w}\biggr)\Biggr)\cdot\biggl(\sum_{e\in E}|\hat{\gamma}(e)|\biggr)\\
		&\geq \sum_{v\sim w}|\hat{\gamma}(e)|\cdot\biggl(\frac{1}{\deg v}+\frac{1}{\deg w}\biggr)\\
		&=\sum_{v\in V}\frac{1}{\deg v}\cdot \biggl(\sum_{e: v \text{ input or output}}|\hat{\gamma}(e)|\biggr)\\
		&\geq \sum_{v\in V}\frac{1}{\deg v}\cdot \biggl|\sum_{e_{\text{in}}: v\text{ input}}\hat{\gamma}(e_{\text{in}})-\sum_{e_{\text{out}}: v\text{ output}}\hat{\gamma}(e_{\text{out}})\biggr|\\
		&=\max_{\gamma:E\rightarrow\mathbb{R},\,\gamma\neq 0}\frac{\sum_{v\in V}\frac{1}{\deg v}\cdot \biggl|\sum_{e_{\text{in}}: v\text{ input}}\gamma(e_{\text{in}})-\sum_{e_{\text{out}}: v\text{ output}}\gamma(e_{\text{out}})\biggr|}{\sum_{e\in E}|\gamma(e)|}.
		\end{align*}This proves the claim.
	\end{proof}As a corollary of Lemma \ref{maintheo1}, we get another characterization of $Q$.
	\begin{cor}
		\begin{equation*}
		Q=\max_{\hat{\Gamma}\subset\Gamma \text{ bipartite}} \frac{\sum_{v\in V}\frac{\deg_{\hat{\Gamma}}(v)}{\deg_\Gamma (v)}}{|E(\hat{\Gamma})|}.
		\end{equation*}
	\end{cor}
	\begin{proof}
		Let's fix $\Gamma'\subset\Gamma$ that maximizes
		\begin{equation*}
		\frac{\sum_{v\in V}\frac{\deg_{\hat{\Gamma}}(v)}{\deg_\Gamma (v)}}{|E(\hat{\Gamma})|}.
		\end{equation*}over all $\hat{\Gamma}\subset\Gamma$ bipartite. Let's fix an orientation and let $\gamma':E(\Gamma)\rightarrow\mathbb{R}$ be $1$ on each oriented edge in $E(\Gamma')$ and $0$ otherwise. Then,
		\begin{align*}
		Q&=\max_{\gamma:E\rightarrow\mathbb{R},\,\gamma\neq 0}\frac{\sum_{v\in V}\frac{1}{\deg v}\cdot \biggl|\sum_{e_{\text{in}}: v\text{ input}}\gamma(e_{\text{in}})-\sum_{e_{\text{out}}: v\text{ output}}\gamma(e_{\text{out}})\biggr|}{\sum_{e\in E}|\gamma(e)|}\\
		&\geq \frac{\sum_{v\in V}\frac{1}{\deg v}\cdot \biggl|\sum_{e_{\text{in}}: v\text{ input}}\gamma'(e_{\text{in}})-\sum_{e_{\text{out}}: v\text{ output}}\gamma'(e_{\text{out}})\biggr|}{\sum_{e\in E}|\gamma'(e)|}\\
		&=\frac{\sum_{v\in V}\frac{\deg_{\Gamma'}(v)}{\deg_\Gamma (v)}}{|E(\Gamma')|}\\
		&=\max_{\hat{\Gamma}\subset\Gamma \text{ bipartite}} \frac{\sum_{v\in V}\frac{\deg_{\hat{\Gamma}}(v)}{\deg_\Gamma (v)}}{|E(\hat{\Gamma})|}.
		\end{align*}To prove the inverse inequality, let $(v_1,v_2)$ be ad edge that maximizes $ \frac{1}{\deg v}+\frac{1}{\deg w}$ over all $(v,w)\in E$. Then, by taking $\hat{\Gamma}\subset \Gamma$ as the bipartite graph containing only the edge $(v_1,v_2)$, we get that
		\begin{equation*}
		\max_{\hat{\Gamma}\subset\Gamma \text{ bipartite}} \frac{\sum_{v\in V}\frac{\deg_{\hat{\Gamma}}(v)}{\deg_\Gamma (v)}}{|E(\hat{\Gamma})|}\geq \frac{1}{\deg v_1}+\frac{1}{\deg v_2}=\max_{v\sim w} \Biggl(\frac{1}{\deg v}+\frac{1}{\deg w}\Biggr)=Q.
		\end{equation*}
		
	\end{proof}
	\subsection{Lower bound for the largest eigenvalue}
	\begin{lem}\label{maintheo2}For every graph,
		\begin{equation*}
		Q\leq \lambda_n.
		\end{equation*}
	\end{lem}
	\begin{proof}As in the proof of Lemma \ref{maintheo1}, fix an edge $(v_1,v_2)$ that maximizes $\frac{1}{\deg v}+\frac{1}{\deg w}$ over all edges $(v,w)$ and let $\gamma':E\rightarrow\mathbb{R}$ be $1$ on $(v_1,v_2)$ and $0$ otherwise. Then, 
		\begin{align*}
		\lambda_n&=\max_{\gamma:E\rightarrow\mathbb{R},\,\gamma\neq 0}\frac{\sum_{v\in V}\frac{1}{\deg v}\cdot \biggl(\sum_{e_{\text{in}}: v\text{ input}}\gamma(e_{\text{in}})-\sum_{e_{\text{out}}: v\text{ output}}\gamma(e_{\text{out}})\biggr)^2}{\sum_{e\in E}\gamma(e)^2}\\
		&\geq \frac{\sum_{v\in V}\frac{1}{\deg v}\cdot \biggl(\sum_{e_{\text{in}}: v\text{ input}}\gamma'(e_{\text{in}})-\sum_{e_{\text{out}}: v\text{ output}}\gamma'(e_{\text{out}})\biggr)^2}{\sum_{e\in E}\gamma'(e)^2}\\
		&=\frac{1}{\deg v_1}+\frac{1}{\deg v_2}\\
		&=Q.
		\end{align*}
	\end{proof}
	\begin{rmk}
		Observe that $Q\geq \frac{n}{n-1}$ if and only if there exists a vertex of degree $1$. In fact, if there exists such a vertex, then
		\begin{equation*}
		Q\geq 1+\frac{1}{n-1}=\frac{n}{n-1}.
		\end{equation*}If there is no such vertex, then
		\begin{equation*}
		Q\leq \frac{1}{2}+\frac{1}{2}=1\leq \frac{n}{n-1}.
		\end{equation*}Therefore, the bound in Lemma \ref{maintheo2} is better than the usual bound $\frac{n}{n-1}\leq\lambda_n$ only for a small class of graphs. However, the aim of our work is not to find the \emph{best possible bounds for $\lambda_n$} but the \emph{best possible bounds for $\lambda_n$ involving $Q$}, in order to show that $Q$ is a Cheeger--like constant. We shall see, in Section \ref{sectionprecisionlower}, that the bound in Lemma \ref{maintheo2} is actually the best possible lower bound for $\lambda_n$ involving only $Q$.
	\end{rmk}
	\subsection{Upper bound for the largest eigenvalue}
	\begin{lem}\label{maintheo3}For every graph,
		\begin{equation*}
		\lambda_n\leq Q\cdot \tau.
		\end{equation*}
	\end{lem}
	\begin{proof}
		We apply \cite[Theorem 5]{upper} to obtain
		\begin{align*}
		\lambda_n&\leq 2-\min_{v\sim w}\frac{\bigl|\mathcal{N}(v)\cap \mathcal{N}(w)\bigr|}{\max\{\deg v,\deg w\}}\\
		&\leq 2-\min_{v\sim w:\,\deg w\geq \deg v}\frac{\deg v + \deg w - n}{\deg w}\\
		&=\max_{v\sim w:\,\deg w\geq \deg v}\frac{\deg w-\deg v+n}{\deg w}\\
		&=\max_{v\sim w:\,\deg w\geq \deg v}\Biggl(\frac{1}{\deg v}+\frac{1}{\deg w}\Biggr)\cdot\Biggl(\frac{(\deg w-\deg v+n)\cdot \deg v}{\deg v+\deg w}\Biggr)\\
		&\leq Q\cdot \tau.
		\end{align*}
	\end{proof}
	Observe that the bound in Lemma \ref{maintheo3} is not a better upper bound for $\lambda_n$ than the one in \cite[Theorem 5]{upper}. Nevertheless, it is a good upper bound for $\lambda_n$ involving $Q$, as we shall see in Section \ref{sectionprecisionupper}.
	\section{Choice of Q}\label{sectionchoiceQ}
	Let us motivate the choice of $Q$. As we have discussed in Section \ref{sectionprel},
	\begin{align}
	\lambda_n&=\max_{f:V\rightarrow\mathbb{R},\, f\neq 0}\frac{\sum_{v\sim w}\biggl(f(v)-f(w)\biggr)^2}{\sum_{v\in V}\deg v\cdot f(v)^2}\label{eqf}\\
	&=\max_{\gamma:E\rightarrow\mathbb{R},\,\gamma\neq 0}\frac{\sum_{v\in V}\frac{1}{\deg v}\cdot \biggl(\sum_{e_{\text{in}}: v\text{ input}}\gamma(e_{\text{in}})-\sum_{e_{\text{out}}: v\text{ output}}\gamma(e_{\text{out}})\biggr)^2}{\sum_{e\in E}\gamma(e)^2}.\label{eqgamma}
	\end{align}We have chosen $Q$ to be the constant that can be written as (\ref{eqgamma}) by replacing the $L_2$--norm by the $L_1$--norm both in the numerator and denominator. We could have chosen to work on the constant that can be written as (\ref{eqf}) by replacing the $L_2$--norm by the $L_1$--norm, but such a constant is actually equal to $1$ for all graphs, as shown by the following lemma. Furthermore, while the characterization of the Cheeger constant is interesting also because it is equal to the second smallest eigenvalue of the \emph{$1$--Laplacian}, one cannot get an analogous constant in this sense because the largest eigenvalue of the \emph{$1$--Laplacian} equals $1$ for every graph, as shown in \cite[Theorem 5.1]{Chung}. For completeness, we shall provide a proof.
	\begin{lem}
		For every graph,
		\begin{equation*}
		\max_{f:V\rightarrow\mathbb{R},\,f\neq 0}\frac{\sum_{v\sim w}\bigl|f(v)-f(w)\bigr|}{\sum_{v\in V}\deg v\cdot \bigl|f(v)\bigr|}=1.
		\end{equation*}
	\end{lem}
	\begin{proof}
		Let $\hat{f}:V\rightarrow\mathbb{R}$ be a maximizer of
		\begin{equation*}
		\frac{\sum_{v\sim w}\bigl|f(v)-f(w)\bigr|}{\sum_{v\in V}\deg v\cdot \bigl|f(v)\bigr|}
		\end{equation*}and assume, without loss of generality, that $\sum_{v\in V}\deg v\cdot \bigl|\hat{f}(v)\bigr|=1$. Then,
		\begin{align*}
		\max_{f:V\rightarrow\mathbb{R},\,f\neq 0}\frac{\sum_{v\sim w}\bigl|f(v)-f(w)\bigr|}{\sum_{v\in V}\deg v\cdot \bigl|f(v)\bigr|}&=\sum_{v\sim w}\bigl|\hat{f}(v)-\hat{f}(w)\bigr|\\
		&\leq \sum_{v\sim w}\bigl|\hat{f}(v)\bigr|+\bigl|\hat{f}(w)\bigr|\\
		&=\sum_{v\in V}\deg v\cdot \bigl|\hat{f}(v)\bigr|\\
		&=1.
		\end{align*}To see the inverse inequality, let $\tilde{f}:V\rightarrow\mathbb{R}$ that is $1$ on a fixed vertex and $0$ on all other vertices. Then,
		\begin{equation*}
		\max_{f:V\rightarrow\mathbb{R},\,f\neq 0}\frac{\sum_{v\sim w}\bigl|f(v)-f(w)\bigr|}{\sum_{v\in V}\deg v\cdot \bigl|f(v)\bigr|}\geq \frac{\sum_{v\sim w}\bigl|\tilde{f}(v)-\tilde{f}(w)\bigr|}{\sum_{v\in V}\deg v\cdot \bigl|\tilde{f}(v)\bigr|}=1.
		\end{equation*}
	\end{proof}
	\section{How good is the lower bound?}\label{sectionprecisionlower}
	To see that $Q\leq \lambda_n$ is a sharp lower bound, consider the case of $K_2$: here, $Q=\lambda_2=2$. Also, for $n>2$, consider a non--bipartite graph such that there exists an edge $(v,w)$ with $\deg v=1$ and $\deg w=2$. Then, clearly
	\begin{equation*}
	Q=1+\frac{1}{2}=\frac{3}{2}
	\end{equation*}and, since the graph is non--bipartite, $\lambda_n<2$. Therefore, if we look for a bound of the form
	\begin{equation*}
	Q\cdot \nu \leq \lambda_n,
	\end{equation*}we must have
	\begin{equation*}
	\nu \leq \frac{\lambda_n}{Q}<\frac{4}{3} \simeq 1.33.
	\end{equation*}Hence $Q\leq \lambda_n$ is actually a good lower bound involving only $Q$ for each $n$.
	\section{How good is the upper bound?}\label{sectionprecisionupper}
	In order to see that the bound $Q\cdot\tau$ is actually a good upper bound for $\lambda_n$, let us first construct an  example for which the bound  $\lambda_n\leq Q\cdot \tau$ is sharp.
	\begin{ex}
		For $d$--regular graphs, it's easy to see that $Q=\frac{2}{d}$ and $\tau=\frac{n}{2}$, therefore $\lambda_n\leq Q\cdot \tau$ is equivalent to
		\begin{equation*}
		\lambda_n\leq \frac{n}{d}.
		\end{equation*}In the particular case of the complete graph $K_n$, $d=n-1$ and $\lambda_n=\frac{n}{n-1}$ \cite{Chung} therefore $\lambda_n=Q\cdot\tau$, i.e. the inequality in Lemma \ref{maintheo3} becomes an equality.
	\end{ex}
	For further motivating our upper bound, we shall:
	\begin{enumerate}
		\item Prove that, for each graph on $n$ nodes,
		\begin{equation*}
		\tau < 0.54\cdot n
		\end{equation*}and $0.54$ is the best $\varepsilon$ with a precision of two decimal places such that
		\begin{equation*}
		\lambda_n\leq Q\cdot \varepsilon \cdot n.
		\end{equation*}
		\item Prove that there is no  bound of the form
		\begin{equation*}
		\lambda_n\leq Q\cdot \Biggl(\frac{n}{2}+c\Biggr),
		\end{equation*}if $c$ is a constant that does not depend on $n$, as we might be tempted to do by looking at the example of regular graphs.
	\end{enumerate}In order to prove these two points, we shall first discuss \emph{one--sided bipartite graphs}, a new big class of graphs that includes among others petal graphs, complete graphs and complete bipartite graphs.
	\subsection{One--sided bipartite graphs}\label{sectionkd}
	\begin{definition}
		Fix $n$ and $k$ such that $0<k\leq n-2$. Let $\Gamma=(V,E)$ be a graph on $n$ vertices such that $V=V_1\sqcup V_2$, $|V_2|=k$ therefore $|V_1|=n-k$, $v_1\sim v_2$ for each $v_1\in V_1$ and $v_2\in V_2$, $\deg v_2=n-k$ for each $v_2\in V_2$ and $\deg v_1=d$ for each $v_1\in V_1$, for some $d\geq k$. Call such a graph a \emph{$(k,d)$--one--sided bipartite graph}.
	\end{definition}
	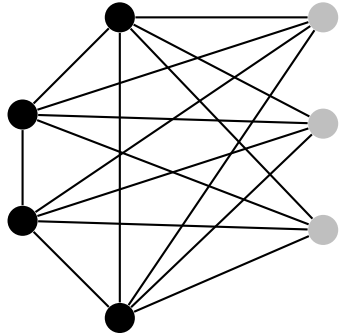
\begin{figure}[ht]
		\centering
		\begin{tikzpicture}
		\node[rednode] (1) {};
		\node[rednode] (2) [below left = of 1]  {};
		\node[rednode] (3) [below = of 2] {};
		\node[] (8) [right = of 1] {};
		\node[rednode] (4) [below right = of 3] {};
		\node[bluenode] (5) [right = of 8]  {};
		\node[bluenode] (6) [below = of 5] {};
		\node[bluenode] (7) [below = of 6] {};
		
		\path[draw,thick]
		(1) edge node {} (2)
		(1) edge node {} (4)
		(1) edge node {} (5)
		(1) edge node {} (6)
		(1) edge node {} (7)
		(2) edge node {} (3)
		(2) edge node {} (5)
		(2) edge node {} (6)
		(2) edge node {} (7)
		(3) edge node {} (4)
		(3) edge node {} (5)
		(3) edge node {} (6)
		(3) edge node {} (7)
		(4) edge node {} (5)
		(4) edge node {} (6)
		(7) edge node {} (4);
		
		\end{tikzpicture}
		\caption{A $(k,d)$--one--sided bipartite graph on $7$ nodes, with $k=3$ and $d=5$. The black nodes are the ones of degree $d$.}
		\label{fig:kd}
	\end{figure}
	\begin{rmk}
		In a $(k,d)$--one--sided bipartite graph, the vertex set is divided into two sets $V_1$ and $V_2$. All possible edges between $V_1$ and $V_2$ are there, the $k$ vertices in $V_2$ are not connected to each other and the vertices in $V_1$ all have degree $d$, therefore there are edges between vertices of $V_1$ if and only if $d>k$ (Figure \ref{fig:kd}). In particular, a $(k,d)$--one--sided bipartite graph is:
		\begin{itemize}
			\item The petal graph if $k=1$ and $d=2$;
			\item The complete graph $K_n$ if $k=1$ and $d=n-1$;
			\item The graph $K_n\setminus\{e\}$ if $k=2$ and $d=n-1$;
			\item The complete bipartite graph $K_{d,n-k}$ if $d=k$;
			\item Not bipartite if $d>k$;
			\item A $d$--regular graph if $d=n-k$.
		\end{itemize}
	\end{rmk}
	\begin{lem}\label{lemexistencekd}
		Given $n$, $k$ and $d$ such that $n\geq 3$, $0<k\leq n-2$ and $k\leq d\leq n-1$, there exists a $(k,d)$--one--sided bipartite graph on $n$ nodes if and only if at least one of $d-k$ and $n-k$ is even.
	\end{lem}
	\begin{proof}
		This follows easily by definition of one--sided bipartite graphs and by \cite[Theorem 2.6]{regular}, which states that a $d$-regular graph on $n$ nodes exists if and only if at least one of $d$ and $n$ is even.
	\end{proof}
	In Theorem \ref{teokd} we shall prove that for a one--sided bipartite graph with $d\geq n-k$,
	\begin{equation*}
	\lambda_n=\frac{d+k}{d}
	\end{equation*}and for a $(k,d)$--one--sided bipartite graph with $d< n-k$,
	\begin{equation*}
	\frac{d+k}{d}\leq\lambda_n\leq\frac{n}{d}.
	\end{equation*} Let's prove a preliminary lemma first.
	\begin{definition}[\cite{JJspectrum}]  
		Given a vertex $v_1$, let $\mathcal{N}(v_1)\subset V$ be the set of  neighbors of $v_1$. We say that $v_1$ and $v_2$ are \emph{duplicate vertices} if $\mathcal{N}(v_1)=\mathcal{N}(v_2)$.
	\end{definition}Observe that, in particular, duplicate vertices have the same degree and they cannot be neighbors of each other.
	\begin{lem}\label{lemduplicate}
		If $v_1$ and $v_2$ are duplicate vertices and $f$ is an eigenfunction for an eigenvalue $\lambda\neq 1$ of $L$,
		\begin{equation*}
		f(v_1)=f(v_2).
		\end{equation*}
	\end{lem}
	\begin{proof}An eigenvalue $\lambda$  of $L$ with eigenfunction $f$ satisfies for each vertex $v$,
		\begin{equation*}
		\lambda\cdot f(v)=L f(v)=f(v)-\frac{1}{\deg v}\cdot\sum_{v'\sim v}f(v').
		\end{equation*}In particular,
		
		\begin{align*}
		\lambda\cdot f(v_i)&=f(v_i)-\frac{1}{\deg v_j}\cdot\sum_{v'\sim v_i}f(v') \text{ for }i,j=1,2.
		\end{align*}
		Therefore,
		\begin{equation*}
		\frac{1}{\deg v_2}\cdot\sum_{v'\sim v_2}f(v')=f(v_1)\cdot (1-\lambda)=f(v_2)\cdot (1-\lambda).
		\end{equation*}Since by assumption $\lambda\neq 1$, this implies that $f(v_1)=f(v_2)$.
	\end{proof}
	
	\begin{theorem}\label{teokd}
		For a $(k,d)$--one--sided bipartite graph with $d\geq n-k$,
		\begin{equation*}
		\lambda_n=\frac{d+k}{d}.
		\end{equation*}For a $(k,d)$--one--sided bipartite graph with $d< n-k$,
		\begin{equation*}
		\frac{d+k}{d}\leq\lambda_n\leq\frac{n}{d}.
		\end{equation*}
	\end{theorem}
	\begin{proof}For any fixed $(k,d)$--one--sided bipartite graph, let $\lambda\neq 0,1$ be an eigenvalue for $L$ with eigenfunction $f$. By construction, in a $(k,d)$--one--sided bipartite graph all $k$ vertices in $V_2$ of degree $n-k$ are duplicate vertices. Therefore, by Lemma \ref{lemduplicate}, $f(v_2)$ is constant for each $v_2\in V_2$. If, in particular, $f(v_2)\neq 0$ for each $v_2\in V_2$, we can define
		\begin{equation*}
		\alpha_{v_2}:=\frac{-\sum_{v_1\in V_1}f(v_1)}{f(v_2)}
		\end{equation*}and, since this is constant for each $v_2\in V_2$, we can write $\alpha_{n-k}= \alpha_{v_2}$. Therefore,
		\begin{equation*}
		\lambda\cdot f(v_2)=f(v_2)-\frac{1}{n-k}\cdot\sum_{v_1\in V_1}f(v_1)=f(v_2)\cdot\biggl(1+\frac{\alpha_{n-k}}{n-k}\biggr),
		\end{equation*}which implies that
		\begin{equation*}
		\lambda=1+\frac{\alpha_{n-k}}{n-k}.
		\end{equation*}In particular, since we are assuming $\lambda\neq 1$, this implies that $\alpha_{n-k}\neq 0$, hence we can write
		\begin{equation*}
		f(v_2)=\frac{-\sum_{v_1\in V_1}f(v_1)}{\alpha_{n-k}}.
		\end{equation*}
		Now, by the orthogonality to the constants, we must have $\sum_v \deg v\cdot f(v)=0$. Hence
		\begin{align*}
		0=&\sum_{v_1\in V_1}d\cdot f(v_1)+k\cdot (n-k)\cdot \biggl(\frac{-\sum_{v_1\in V_1}f(v_1)}{\alpha_{n-k}}\biggr)\\&=\Biggl(\sum_{v_1\in V_1}f(v_1)\Biggr)\cdot\Biggl(d-\frac{k\cdot(n-k)}{\alpha_{n-k}}\Biggr).
		\end{align*}If
		\begin{equation*}
		\sum_{v_1\in V_1}f(v_1)=0,
		\end{equation*}then $\alpha_{n-k}=0$ therefore $\lambda=1$, which is a contradiction. Therefore we must have
		\begin{equation*}
		d-\frac{k\cdot(n-k)}{\alpha_{n-k}}=0,
		\end{equation*}which implies that 
		\begin{equation*}
		\alpha_{n-k}=\frac{k\cdot(n-k)}{d}
		\end{equation*}therefore
		\begin{equation*}
		\lambda=1+\frac{k}{d}=\frac{d+k}{d}.
		\end{equation*}This proves that $\frac{d+k}{d}$ is an eigenvalue, therefore 
		\begin{equation*}
		\lambda_n\geq \frac{d+k}{d}.
		\end{equation*}Now, in the particular case of $d\geq n-k$, we can prove also the inverse inequality by applying \cite[Theorem 5]{upper}, which states that
		\begin{equation*}
		\lambda_n\leq 2-\min_{v\sim w}\frac{\bigl|\mathcal{N}(v)\cap \mathcal{N}(w)\bigr|}{\max\{\deg v,\deg w\}}.
		\end{equation*}Let's prove that, for a $(k,d)$--one--sided bipartite graph with $d\geq n-k$,
		\begin{equation*}
		\min_{v\sim w}\frac{\bigl|\mathcal{N}(v)\cap \mathcal{N}(w)\bigr|}{\max\{\deg v,\deg w\}}=\frac{d-k}{d}.
		\end{equation*}Let's consider the possible cases.
		\begin{itemize}
			\item Case 1: $v\in V_1$ and $w\in V_2$. Since we are assuming $d\geq n-k$, we have that $\max\{\deg v,\deg w\}=d$. Therefore,
			\begin{equation*}
			\frac{\bigl|\mathcal{N}(v)\cap \mathcal{N}(w)\bigr|}{\max\{\deg v,\deg w\}}=\frac{d-k}{d}.
			\end{equation*}
			\item Case 2: $v,w\in V_1$. In this case, $\deg v=\deg w=d$. Also, $v$ and $w$ have $k$ neighbors in common in $V_2$ and at least $2(d-k)-(n-k)$ neighbors in common in $V_1$. Therefore,
			\begin{equation*}
			\frac{\bigl|\mathcal{N}(v)\cap \mathcal{N}(w)\bigr|}{\max\{\deg v,\deg w\}}\geq \frac{k+2 (d-k)-(n-k)}{d}=\frac{2d-n}{d}\geq \frac{d-k}{d},
			\end{equation*}where the last inequality follows from the assumption that $d\geq n-k$.
		\end{itemize}Therefore,
		\begin{equation*}
		\min_{v\sim w}\frac{\bigl|\mathcal{N}(v)\cap \mathcal{N}(w)\bigr|}{\max\{\deg v,\deg w\}}=\frac{d-k}{d}
		\end{equation*}and by \cite[Theorem 5]{upper} this implies that
		\begin{equation*}
		\lambda_n\leq 2-\frac{d-k}{d}=\frac{d+k}{d},
		\end{equation*}therefore that the equality holds in this case.\newline 
		
		It remains to prove that, for $d<n-k$,
		\begin{equation}\label{b20}
		\lambda_n\leq \frac{n}{d}.
		\end{equation}Let again $\lambda\neq 0,1$ be an eigenvalue for $L$ with eigenfunction $f$. We know that $f(v_2)$ must be constant for each $v_2\in V_2$. In particular, if $f(v_2)\neq 0$, as shown in the first part of the proof we have that
		\begin{equation*}
		\lambda=\frac{d+k}{d}.
		\end{equation*}Therefore, since we are assuming $d<n-k$, we have that
		\begin{equation*}
		\lambda<\frac{n}{d}.
		\end{equation*}
		Let's now consider the case $f(v_2)=0$. We have that
		\begin{align*}
		\lambda&=\frac{\sum_{v\sim w}\biggl(f(v)-f(w)\biggr)^2}{\sum_{v_1\in V}d \cdot f(v_1)^2}\\
		&=\frac{\sum_{v_1\in V_1}k\cdot f(v_1)^2+\sum_{v\sim w;v,w \in V_1}\biggl(f(v)-f(w)\biggr)^2}{\sum_{v_1\in V}d \cdot f(v_1)^2}\\
		&=\frac{k}{d}+\frac{\sum_{v\sim w;v,w \in V_1}\biggl(f(v)-f(w)\biggr)^2}{\sum_{v_1\in V}d \cdot f(v_1)^2}\\
		&\leq \frac{k}{d}+\frac{d-k}{d}\lambda'_n,
		\end{align*}where $\lambda'_n$ is the largest eigenvalue of a $(d-k)$--regular graph on $n-k$ nodes, therefore
		\begin{equation*} 
		\lambda'_n\leq\frac{n-k}{d-k}.
		\end{equation*}In fact, in order to prove \eqref{b20}, it suffices to show that, for $\hat{d}$--regular graphs on $\hat{n}$ nodes, the largest eigenvalue of the non--normalized Laplace operator is at most $\hat{n}$. This is actually true for every graph, because for the non--normalized Laplacian the complete graph has largest eigenvalue equal to $\hat{n}$ and, \emph{if an edge is added into a graph, then none of its Laplacian eigenvalues can decrease} \cite{nonnormalized}. Therefore,
		\begin{equation*}
		\lambda\leq \frac{k}{d}+\frac{d-k}{d}\lambda'_n \leq \frac{k}{d}+ \frac{d-k}{d}\cdot \frac{n-k}{d-k}=\frac{n}{d}.
		\end{equation*}This proves that any eigenvalue of $L$, in the case when $d<n-k$, is at most $n/d$. Therefore, in particular, $\lambda_n\leq n/d$.
	\end{proof}
	
	\begin{rmk}
		Observe also that, for $(k,d)$--one--sided bipartite graphs with $d\geq n-k$,
		\begin{equation*}
		Q=\frac{1}{d}+\frac{1}{n-k}.
		\end{equation*}For $(k,d)$--one--sided bipartite graphs with $d< n-k$,
		\begin{equation*}
		Q=\frac{2}{d}.
		\end{equation*}
	\end{rmk}
	\subsection{Conclusions}
	As a consequence of Theorem \ref{teokd}, we can prove the following corollary that further motivates the upper bound in Lemma \ref{maintheo3}.
	\begin{cor}
		\begin{enumerate}
			\item For each graph on $n$ nodes,
			\begin{equation*}
			\tau < 0.54\cdot n
			\end{equation*}and $0.54$ is the best $\varepsilon$ with a precision of two decimal places such that
			\begin{equation*}
			\lambda_n\leq Q\cdot \varepsilon \cdot n.
			\end{equation*}
			\item We can not have a bound of the form
			\begin{equation*}
			\lambda_n\leq Q\cdot \Biggl(\frac{n}{2}+c\Biggr),
			\end{equation*}if $c$ is a constant that does not depend on $n$.
		\end{enumerate}
	\end{cor}
	\begin{proof}
		\begin{enumerate}
			\item By writing in WolframAlpha \cite{wolframalpha}:
			\begin{verbatim}
			(y(z-y+x))/(y+z) >= 0.54 x 
			with x>0, y>0, y<x, z>=y, z<x, integer solutions
			\end{verbatim} 
			one can see that there is no solution. Therefore,
			\begin{equation*}
			\tau < 0.54\cdot n
			\end{equation*}for each graph and, by Lemma \ref{maintheo3}, $\lambda_n\leq Q\cdot 0.54\cdot n$. In order to see that $0.54$ is the best $\varepsilon$ with a precision of two decimal places such that
			\begin{equation*}
			\lambda_n\leq Q\cdot \varepsilon \cdot n,
			\end{equation*}observe that for $(k,d)$--one--sided bipartite graphs with $d\geq n-k$, we have that
			\begin{equation*}
			\frac{\lambda_n}{Q}=\frac{dn-dk+kn-k^2}{d+n-k}.
			\end{equation*}By writing in WolframAlpha \cite{wolframalpha}:
			\begin{verbatim}
			(xz-yz+xy-y^2)/(x-y+z) > (0.53*x),
			with x>0, y>0, y<x-1, z>=x-y, z>=y, z<x integer solutions
			\end{verbatim} 
			one can see that there are solutions, for example for $x=n=249$, $y=k=69$ and $z=d=241$. Since $n-k$ is even, by Lemma \ref{lemexistencekd} there exists a $(k,d)$--one--sided bipartite graph with these values of $n$, $k$ and $d$. For such a graph,
			\begin{equation*}
			\lambda_n>Q\cdot 0.53\cdot n.
			\end{equation*}This proves the first claim.
			\item For $(k,d)$--one--sided bipartite graphs with $d=n-1$ and $k=\frac{n}{4}$,
			\begin{equation*}
			\frac{\lambda_n}{Q}=\frac{15n^2-12n}{28n-16}.
			\end{equation*}Therefore, if we look for an upper bound of $\lambda_n$ of the form $Q\cdot g(n)$, we must have $g(n)\geq \frac{15n^2-12n}{28n-16}$ for each $n$. In particular, we can not take any $g(n)=\frac{n}{2}+c$ if $c$ is a constant that does not depend on $n$.
		\end{enumerate}
	\end{proof}

	\subsection*{Data Availability Statement}Data sharing is not applicable to this article as no new data were created or analyzed in this study.
	
	\subsection*{Acknowledgements}Raffaella Mulas wants to thank Florentin Münch and Emil Saucan for the helpful comments and discussions.
	
	\bibliographystyle{alpha}
	\bibliography{CheegerLike20.01.21}	

\begin{thebibliography}{JMM21}

\bibitem[AM85]{alon}
N.~Alon and V.D. Milman.
\newblock $\lambda$1, isoperimetric inequalities for graphs, and
  superconcentrators.
\newblock {\em Journal of Combinatorial Theory, Series B}, 38(1):73--88, 1985.

\bibitem[Bau12]{Bauer}
F.~Bauer.
\newblock {Normalized graph Laplacians for directed graphs}.
\newblock {\em Linear Algebra and its Applications}, 436:4193--4222, 2012.

\bibitem[BHJ14]{dual}
F.~Bauer, B.~Hua, and J.~Jost.
\newblock {The dual Cheeger constant and spectra of infinite graphs}.
\newblock {\em Advances in Mathematics}, 251:147--194, 2014.

\bibitem[BJ08]{JJspectrum}
A.~Banerjee and J.~Jost.
\newblock {On the spectrum of the normalized graph Laplacian}.
\newblock {\em Linear Algebra and its Applications}, 428:3015--3022, 2008.

\bibitem[BJ13]{dual2}
F.~Bauer and J.~Jost.
\newblock {Bipartite and neighborhood graphs and the spectrum of the normalized
  graph Laplacian}.
\newblock {\em Communications in Analysis and Geometry}, 21:787–--845, 2013.

\bibitem[Cha09]{Chang2}
K.C. Chang.
\newblock {The spectrum of the 1--Laplace operator}.
\newblock {\em Comm. Contemporary Mathematics}, 11:865--894, 2009.

\bibitem[Cha16]{Chang}
K.C. Chang.
\newblock {Spectrum of the 1--Laplacian and Cheeger's Constant on Graphs}.
\newblock {\em Journal of Graph Theory}, 81:167--207, 2016.

\bibitem[Chu97]{Chung}
F.~Chung.
\newblock Spectral graph theory.
\newblock {\em American Mathematical Society}, 1997.

\bibitem[CSZ15]{chang2016a}
K.C. Chang, S.~Shao, and D.~Zhang.
\newblock {The $1$-Laplacian Cheeger Cut: Theory and Algorithms}.
\newblock {\em Journal of Computational Mathematics}, (33):443--467, 2015.

\bibitem[CSZ16]{chang2016b}
K.C. Chang, S.~Shao, and D.~Zhang.
\newblock {Spectrum of the signless 1-Laplacian and the dual Cheeger constant
  on graphs}.
\newblock arXiv:1607.00489, 2016.

\bibitem[CZ12]{regular}
G.~Chartrand and P.~Zhang.
\newblock {\em A First Course in Graph Theory}.
\newblock Dover Publications, 2012.

\bibitem[Dod84]{dodziuk}
J.~Dodziuk.
\newblock Difference equations, isoperimetric inequality and transience of
  certain random walks.
\newblock {\em Transactions of the American Mathematical Society},
  284(2):787--794, 1984.

\bibitem[HB10]{Hein2}
M.~Hein and T.~Bühler.
\newblock {An inverse power method for nonlinear eigenproblems with
  applications in 1--spectral clustering and sparse PCA}.
\newblock {\em NIPS}, pages 847--855, 2010.

\bibitem[HS11]{Hein}
M.~Hein and S.~Setzer.
\newblock {Beyond spectral clustering -- tight relaxations of balanced graph
  cuts}.
\newblock {\em Advances in Neural Information Processing Systems},
  24:2366--2374, 2011.

\bibitem[JM19]{Hypergraphs}
J.~Jost and R.~Mulas.
\newblock {Hypergraph Laplace operators for chemical reaction networks}.
\newblock {\em Advances in Mathematics}, 351:870--896, 2019.

\bibitem[JMM21]{jmm}
J.~Jost, R.~Mulas, and F.~M\"unch.
\newblock {Spectral gap of the largest eigenvalue of the normalized graph
  Laplacian}.
\newblock {\em Communications in Mathematics and Statistics}, 2021.
\newblock To appear. DOI :10.1007/s40304-020-00222-7.

\bibitem[Kir05]{nonnormalized}
S.~Kirkland.
\newblock {Completion of Laplacian integral graphs via edge addition}.
\newblock {\em Discrete Mathematics}, 295:75--90, 2005.

\bibitem[RS13]{upper}
O.~Rojo and R.L. Soto.
\newblock {A new upper bound on the largest normalized Laplacian eigenvalue}.
\newblock {\em Operators and matrices}, 7:323--332, 2013.

\bibitem[{Wol}]{wolframalpha}
{Wolfram Alpha LLC}.
\newblock {WolframAlpha}.
\newblock https://www.wolframalpha.com.

\end{thebibliography}
\end{document}